\documentclass{amsproc}
\usepackage{amssymb}
\usepackage{amsfonts}

\theoremstyle{plain}

\newtheorem{definition}{Definition}
\newtheorem{example}{Example}

\newtheorem{proposition}{Proposition}
\newtheorem{remark}{Remark}

\newtheorem{theorem}{Theorem}
\numberwithin{equation}{section}
\input{tcilatex}

\begin{document}
\title[Similar Curves]{Generalized Similar Frenet Curves}
\author{Fatma G\"{O}K\c{C}EL\.{I}K}
\address{Department of Mathematics, Faculty of Science, University of Ankara
Tandogan, Ankara, TURKEY }
\email{fgokcelik@ankara.edu.tr}
\author{Seher KAYA}
\address{Department of Mathematics, Faculty of Science, University of Ankara
Tandogan, Ankara, TURKEY}
\email{seherkaya@ankara.edu.tr}
\author{Yusuf YAYLI}
\address{Department of Mathematics, Faculty of Science, University of Ankara
Tandogan, Ankara, TURKEY}
\email{yyayli@science.ankara.edu.tr}
\author{F. Nejat EKMEKC\.{I}}
\address{Department of Mathematics, Faculty of Science, University of Ankara
Tandogan, Ankara, TURKEY}
\email{ekmekci@science.ankara.edu.tr}
\date{March 10, 2014}
\thanks{}

\begin{abstract}
The paper is devoted to differential geometric invariants determining a
Frenet curve in $E^{n}$ up to a direct similarity$.$ These invariants can be
presented by the Euclidean curvatures in terms of an arc lengths of the
spherical indicatrices. Then, these invariants expressed by focal curvatures
of the curve. And then, we give the relationship between curvatures of
evolute curve and shape curvatures. Morever, these invariants is given the
geometric interpretation.
\end{abstract}

\maketitle

\section{Introduction}

Curves are important for many areas of science. In Physics, the particle
orbits as it moves is determined by using the curves. Especially, the curves
have an important place in fractal science. In this field, fractal curves
variation based on changing the similarities mapping on the same segments.
In the human body,fingerprint is the remarkable example of the fractal
curve. Fingerprint is occurs the similar curves. In nature, fractal curve
which is called Koch curve is seen in the structure of snowflake. Koch curve
is constructed by using the base curve and its similar curves.%
\begin{equation*}
\FRAME{itbpF}{4.9536in}{0.9158in}{0in}{}{}{Figure}{\special{language
"Scientific Word";type "GRAPHIC";maintain-aspect-ratio TRUE;display
"USEDEF";valid_file "T";width 4.9536in;height 0.9158in;depth
0in;original-width 7.1036in;original-height 1.2903in;cropleft "0";croptop
"1";cropright "1";cropbottom "0";tempfilename
'N228BN00.wmf';tempfile-properties "XPR";}}
\end{equation*}%
\begin{equation*}
\text{Figure 1. The construction of Koch curve.}
\end{equation*}

Euclidean Geometry can be described as a study of the properties of
geometric figures. Only the properties which do not change under isometries
deserve to be called geometric properties. A similarity of the Euclidean
space $E^{n}$ is an automorphism of $E^{n}$ for which the ratio: distance
between two arbitrary points to distance between the transformed points is a
positive constant. This transformation preserves angles. In this study, we
investigate preserve which properties of the curves under similarity
transformation. The arc length parameter of the curve is not protected up to
similarity transformation but arc length parameter of indicatrix curves are
preserved. So, the curve is re-parameterized in terms of arc length
parameters of indicatrix curves.

Encheva and Georgiev used spherical tangent indicatrix of the curve and its
arc length parameter. As general of this study, we use all spherical
indicatrix curves and their arc length parameters that is, we use its
spherical images $V_{i}$ which is the $i-th$ Frenet vector field in $E^{n}$.
After, we calculated some differential-geometric invariants of curves up to
direct similarities. The invariants are introduced shape curvatures. Then,
we determine a curve which given the shape curvatures in the odd-dimensional
and even-dimensional Euclidean space and this is illustrated with an example
in $E^{3}.$ Finally, a geometric interpretation of shape curvatures are
given in $E^{3}$.

\section{Preliminaries}

In this section, we review some basic concepts on classical differential
geometry of space curves in Euclidean $n-$space. For any two vectors $%
x=(x_{1},x_{2},...,x_{n})$ and $y=(y_{1},y_{2},...,y_{n})\in E^{n}$, $x.y$
as the standard inner product. Let $\alpha :I\subset \mathbb{R}\rightarrow
E^{n}$ be a curve with $\dot{\alpha}(t)\neq 0,$ where $\dot{\alpha}(t)=%
\dfrac{d\alpha }{dt}$. We also denote the norm of $x$ by $\left \Vert
x\right \Vert .$ The arc length parametres of curve $\alpha $ is determined
such that $\left \Vert \alpha ^{\prime }(s)\right \Vert =1,$ where $\alpha
^{\prime }(s)=\dfrac{d\alpha }{ds}.$ Let $V_{1},$ $V_{2},$..., $V_{n}$ be a
Frenet moving $n-$frame of the curve $\alpha .$ Then the following
Frenet-Serret formula holds%
\begin{eqnarray}
V_{1}^{\prime }(s) &=&\kappa _{1}(s)V_{2}(s)  \label{1} \\
V_{i}^{\prime }(s) &=&-\kappa _{i-1}(s)V_{i-1}(s)+\kappa _{i}(s)V_{i+1}(s) 
\notag \\
V_{n}^{\prime }(s) &=&-\kappa _{n-1}(s)V_{n-1}(s)  \notag
\end{eqnarray}%
where $\kappa _{1},$ $\kappa _{2},...,\kappa _{n-1}$ are the curvatures of
the curve $\alpha $ at $s$.

We will study the differential-geometric invariants of a curve in $E^{n}$
with respect to the group $Sim^{+}(\mathbb{R}^{n})$ of all
orientation-preserving similarities of $\mathbb{R}^{n}.$ Any such similarity 
$F:\mathbb{R}^{n}\rightarrow \mathbb{R}^{n}$ is called a \textit{direct
similarity }and can be expressed in the form%
\begin{equation*}
F(x)=\lambda Ax+b
\end{equation*}%
where $x\in E^{n}$ is an arbitrary point, $A$ is an orthogonal $n\times n$
matrix, $b$ is a translation vector and $\lambda >0$ is a real constant.

We denote the image of the curve $\alpha $ under the direct similarity $F$
by the curve $\bar{\alpha},$ i.e., $\bar{\alpha}=F\circ \alpha .$ Then, the
curve $\bar{\alpha}$ can be expressed as 
\begin{equation}
F\circ \alpha (t)=F(\alpha (t))=\lambda A\alpha (t)+b.  \label{2}
\end{equation}%
The arc length functions of $\alpha $ and $\bar{\alpha}$ are 
\begin{equation}
s(t)=\dint \limits_{t_{0}}^{t}\left \Vert \frac{d\alpha (u)}{du}\right \Vert
du\text{ \ \ \ and \ \ \ }\bar{s}(t)=\dint \limits_{t_{0}}^{t}\left \Vert 
\frac{d\bar{\alpha}(u)}{du}\right \Vert du=\lambda s(t).  \label{3}
\end{equation}%
Let $\left \{ \bar{V}_{1},\bar{V}_{2},...,\bar{V}_{n},\bar{\kappa}_{1},\bar{%
\kappa}_{2},...,\bar{\kappa}_{\left( n-1\right) }\right \} $ be a Frenet
apparatus of the curve $\bar{\alpha}.$ Since $\dfrac{ds}{d\bar{s}}=\dfrac{1}{%
\lambda }(=$const.$)$ the curvatures of the curve $\bar{\alpha}$ are given
by 
\begin{equation}
\bar{\kappa}_{i}=\dfrac{1}{\lambda }\kappa _{i}(s)\text{ \ \ \ ,\ \ \ \ \ }%
i=1,2,...,n-1.  \label{4}
\end{equation}%
We obtain $\kappa _{i}ds=\bar{\kappa}_{i}d\bar{s}.$

\section{Expressed with respect to arc length parameter $\protect\sigma $ of
the $V_{i}-$ indicatrix curve of the curve $\protect\alpha $}

In this section, we give some characterizations of the curve $\alpha $ by
using the arc length parameters of its $V_{i}$ indicatrix curve.

Let $\gamma (\sigma _{i})=V_{i}(s)$ be the spherical $V_{i}$ indicatrix of
the curve $\alpha $ and be $\sigma _{i}$ an arc length parameter of the
curve $\gamma .$ Then the curve $\alpha $ admits a reparametrization by $%
\sigma _{i}$%
\begin{equation*}
\alpha =\alpha (\sigma _{i}):I\subset \mathbb{R}\rightarrow \mathbb{E}^{n}.
\end{equation*}

It is clear that 
\begin{equation}
d\sigma _{i}=\sqrt{\left( \kappa _{i-1}(s)\right) ^{2}+\left( \kappa
_{i}(s)\right) ^{2}}ds\text{ \ \ \ \ , \ \ \ \ }\dfrac{d}{d\sigma _{i}}=%
\dfrac{1}{\sqrt{\left( \kappa _{i-1}\right) ^{2}+\left( \kappa _{i}\right)
^{2}}}\dfrac{d}{ds}.  \label{6}
\end{equation}%
Hence, $d\sigma _{i}=\sqrt{\left( \kappa _{i-1}\right) ^{2}+\left( \kappa
_{i}\right) ^{2}}ds$ is invariant under the group of the direct similarities
of $E^{n}.$

Let $V_{1},$ $V_{2},$..., $V_{n}$ be a Frenet frame field along the curve $%
\alpha $ parameterized by the arc length parameter $\sigma _{i}$ of its $%
V_{i}-$indicatrix curve. Then the structure equations of the curve $\alpha $
are given by%
\begin{equation}
\dfrac{d\alpha }{d\sigma _{i}}=\dfrac{1}{\sqrt{\left( \kappa _{i-1}\right)
^{2}+\left( \kappa _{i}\right) ^{2}}}V_{1}(s),  \label{7}
\end{equation}

\begin{equation}
\dfrac{d}{d\sigma _{i}}%
(V_{1},V_{2},...,V_{n-1},V_{n})^{T}=K(V_{1},V_{2},...,V_{n-1},V_{n})^{T}
\label{8}
\end{equation}

where,

\begin{equation*}
K=\left( 
\begin{array}{ccccc}
{\tiny 0} & \tfrac{\kappa _{1}}{\sqrt{\left( \kappa _{i-1}\right)
^{2}+\left( \kappa _{i}\right) ^{2}}} & {\tiny \cdots } & {\tiny 0} & {\tiny %
0} \\ 
{\tiny -}\tfrac{\kappa _{1}}{\sqrt{\left( \kappa _{i-1}\right) ^{2}+\left(
\kappa _{i}\right) ^{2}}} & {\tiny 0} & {\tiny \cdots } & {\tiny 0} & {\tiny %
0} \\ 
{\tiny 0} & {\tiny -}\tfrac{\kappa _{2}}{\sqrt{\left( \kappa _{i-1}\right)
^{2}+\left( \kappa _{i}\right) ^{2}}} & {\tiny \cdots } & {\tiny 0} & {\tiny %
0} \\ 
{\tiny 0} & {\tiny 0} & {\tiny \cdots } & {\tiny 0} & {\tiny 0} \\ 
{\tiny \vdots } & {\tiny \vdots } &  & {\tiny \vdots } & {\tiny \vdots } \\ 
{\tiny 0} & {\tiny 0} & {\tiny \cdots } & {\tiny 0} & \tfrac{\kappa _{n-1}}{%
\sqrt{\left( \kappa _{i-1}\right) ^{2}+\left( \kappa _{i}\right) ^{2}}} \\ 
{\tiny 0} & {\tiny 0} & {\tiny \cdots } & {\tiny -}\tfrac{\kappa _{n-1}}{%
\sqrt{\left( \kappa _{i-1}\right) ^{2}+\left( \kappa _{i}\right) ^{2}}} & 
{\tiny 0}%
\end{array}%
\right)
\end{equation*}%
\begin{equation*}
\left \{ \dfrac{1}{\sqrt{\left( \kappa _{i-1}\right) ^{2}+\left( \kappa
_{i}\right) ^{2}}}V_{1}(s),\dfrac{1}{\sqrt{\left( \kappa _{i-1}\right)
^{2}+\left( \kappa _{i}\right) ^{2}}}V_{2}(s),...,\dfrac{1}{\sqrt{\left(
\kappa _{i-1}\right) ^{2}+\left( \kappa _{i}\right) ^{2}}}V_{n}(s)\right \}
\end{equation*}%
are orthogonal n-frame of the curve $\alpha (\sigma _{i}).$

\begin{equation*}
\text{We take \ }\tilde{\kappa}=\dfrac{d}{ds}(\dfrac{1}{\sqrt{\left( \kappa
_{i-1}\right) ^{2}+\left( \kappa _{i}\right) ^{2}}})\text{ \ and }\tilde{%
\kappa}_{j}=\dfrac{\kappa _{j}}{\sqrt{\left( \kappa _{i-1}\right)
^{2}+\left( \kappa _{i}\right) ^{2}}},\ \ j=1,2,\ldots ,n-1.
\end{equation*}%
So we can write,%
\begin{equation*}
\widetilde{K}=\left( 
\begin{array}{ccccccc}
\tilde{\kappa} & \tilde{\kappa}_{1} & 0 & \cdots & 0 & 0 & 0 \\ 
-\tilde{\kappa}_{1} & \tilde{\kappa}\  & \tilde{\kappa}_{2} & \cdots & 0 & 0
& 0 \\ 
0 & -\tilde{\kappa}_{2} & \tilde{\kappa} & \cdots & 0 & 0 & 0 \\ 
0 & 0 & -\tilde{\kappa}_{3} & \cdots & 0 & 0 & 0 \\ 
\vdots & \vdots & \vdots &  & \vdots & \vdots & \vdots \\ 
0 & 0 & 0 & \cdots & -\tilde{\kappa}_{n-2} & \tilde{\kappa} & \tilde{\kappa}%
_{n-1} \\ 
0 & 0 & 0 & \cdots & 0 & -\tilde{\kappa}_{n-1} & \tilde{\kappa}%
\end{array}%
\right)
\end{equation*}%
Then we obtain this equation%
\begin{equation}
\dfrac{d}{d\sigma _{i}}\left( \tfrac{1}{\sqrt{\left( \kappa _{i-1}\right)
^{2}+\left( \kappa _{i}\right) ^{2}}}{\tiny V}_{1}{\tiny ,...,}\tfrac{1}{%
\sqrt{\left( \kappa _{i-1}\right) ^{2}+\left( \kappa _{i}\right) ^{2}}}%
{\tiny V}_{n}\right) ^{T}=\widetilde{K}\left( \tfrac{1}{\sqrt{\left( \kappa
_{i-1}\right) ^{2}+\left( \kappa _{i}\right) ^{2}}}{\tiny V}_{1}{\tiny ,...,}%
\tfrac{1}{\sqrt{\left( \kappa _{i-1}\right) ^{2}+\left( \kappa _{i}\right)
^{2}}}{\tiny V}_{n}\right) ^{T}  \label{9}
\end{equation}

\begin{remark}
For $i=1$ $($that is, $\sigma _{i}=\sigma )$ equation (3.3) coincidence with
equation (2.6) in $\left[ 8\right] $.
\end{remark}

\begin{definition}
Let $\alpha :I\subset \mathbb{R}\rightarrow E^{n}$ be a Frenet space curve
parameterized by an arc length parameter $\sigma _{i}$ of its $V_{i}-$%
indicatrix curve. The functions 
\begin{equation}
\tilde{\kappa}(\sigma _{i})=-\dfrac{d\sqrt{\left( \kappa _{i-1}\right)
^{2}+\left( \kappa _{i}\right) ^{2}}}{\sqrt{\left( \kappa _{i-1}\right)
^{2}+\left( \kappa _{i}\right) ^{2}}d\sigma _{i}}\text{ \ \ and \ \ }\tilde{%
\kappa}_{j}(\sigma _{i})=\dfrac{\kappa _{j}}{\sqrt{\left( \kappa
_{i-1}\right) ^{2}+\left( \kappa _{i}\right) ^{2}}}\text{ \ , \ \ }%
j=1,2,...,n-1  \label{10}
\end{equation}%
are called shape curvatures of the curve $\alpha .$
\end{definition}

\begin{proposition}
\label{Prop.1}Let $\alpha (\sigma _{i}):I\subset \mathbb{R}\rightarrow 
\mathbb{E}^{n}$ Frenet curve the orthogonal frame is \newline
$\left\{ \dfrac{1}{\sqrt{\left( \kappa _{i-1}\right) ^{2}+\left( \kappa
_{i}\right) ^{2}}}V_{1},\dfrac{1}{\sqrt{\left( \kappa _{i-1}\right)
^{2}+\left( \kappa _{i}\right) ^{2}}}V_{2},...,\dfrac{1}{\sqrt{\left( \kappa
_{i-1}\right) ^{2}+\left( \kappa _{i}\right) ^{2}}}V_{n}\right\} .$ \newline
The functions 
\begin{equation*}
\tilde{\kappa}(\sigma _{i})=-\dfrac{d\sqrt{\left( \kappa _{i-1}\right)
^{2}+\left( \kappa _{i}\right) ^{2}}}{\sqrt{\left( \kappa _{i-1}\right)
^{2}+\left( \kappa _{i}\right) ^{2}}d\sigma _{i}}\text{ \ \ and \ \ \ }%
\tilde{\kappa}_{j}(\sigma _{i})=\dfrac{\kappa _{j}}{\sqrt{\left( \kappa
_{i-1}\right) ^{2}+\left( \kappa _{i}\right) ^{2}}}\text{ \ , \ \ }%
j=1,2,...,n-1\text{\ }\ 
\end{equation*}%
are differential geometric invariants determining the curve $\alpha $ up to
a direct similarity.
\end{proposition}

\begin{proof}
Let $\left \{ \bar{V}_{1},\bar{V}_{2},...,\bar{V}_{n},\bar{\kappa}_{1},\bar{%
\kappa}_{2},...,\bar{\kappa}_{\left( n-1\right) }\right \} $ be a Frenet
apparatus of the curve $\bar{\alpha}=F\circ \alpha .$ The shape curvatures
of the curve $\bar{\alpha}$ are given

\begin{equation*}
\widetilde{\bar{\kappa}}(\sigma _{i})=-\dfrac{d\sqrt{\left( \bar{\kappa}%
_{i-1}\right) ^{2}+\left( \bar{\kappa}_{i}\right) ^{2}}}{\sqrt{\left( \bar{%
\kappa}_{i-1}\right) ^{2}+\left( \bar{\kappa}_{i}\right) ^{2}}d\bar{\sigma}%
_{i}}
\end{equation*}

By using the $\bar{\kappa}_{i}=\dfrac{1}{\lambda }\kappa _{i}(s)$ \ ,\ \ \ $%
i=1,2,...,n-1$ and $\dfrac{ds}{d\bar{s}}=\dfrac{1}{\lambda }(=$const.$)$ \
we have

\begin{equation*}
\widetilde{\bar{\kappa}}(\sigma _{i})=\tilde{\kappa}(\sigma _{i})
\end{equation*}

and are obtained analogously%
\begin{eqnarray*}
\widetilde{\bar{\kappa}_{j}}(\sigma _{i}) &=&\dfrac{\bar{\kappa}_{j}}{\sqrt{%
\left( \bar{\kappa}_{i-1}\right) ^{2}+\left( \bar{\kappa}_{i}\right) ^{2}}}
\\
\widetilde{\bar{\kappa}_{j}}(\sigma _{i}) &=&\tilde{\kappa}_{j}(\sigma _{i})
\end{eqnarray*}%
The proof is completed.
\end{proof}

\begin{remark}
For $i=1$ $($that is$,$ $\sigma _{i}=\sigma )$ $\tilde{\kappa}$ and $\tilde{%
\kappa}_{j}$ $(j=2,\ldots ,n-1)$ coincidence with $\tilde{\kappa}_{1}$ and $%
\tilde{\kappa}_{i}$ $(i=2,3,\ldots ,n-1)$ respectively in $\left[ 8\right] .$
\end{remark}

\section{The relation between the curve $\protect\alpha $ and its evolute
curve}

\begin{definition}
Let $\beta :I\subset \mathbb{R}\rightarrow E^{3}$ be a unit speed Frenet
with Serret-Frenet apparatus $\left\{ \kappa _{1},\kappa
_{2},V_{1},V_{2},V_{3}\right\} $ and be an evolute curve of $\alpha $. Then
the following equality%
\begin{equation*}
\beta (s)=\alpha (s)+m_{1}(s)V_{2}(s)+m_{2}(s)V_{3}(s)
\end{equation*}%
where $m_{1}(s)=\dfrac{1}{\kappa _{1}(s)},$ $m_{2}(s)=\dfrac{1}{\kappa
_{1}(s)}\cot \left( \dint \kappa _{2}(s)ds\right) .$
\end{definition}

In \cite{Uribe}$,$ R. Uribe-Vargas found formulas which express the
Euclidean curvatures in terms of the focal curvatures. Similarly, we may
represent all differential-geometric invariants $\tilde{\kappa}_{i}$ by the $%
m_{i}$ $(i=1,2)$ curvatures and their derivatives.For $i=1$ we obtain that $%
\tilde{\kappa}_{1}=\dfrac{d}{ds}(\dfrac{1}{\kappa _{1}})$ and $\tilde{\kappa}%
_{2}=\dfrac{\kappa _{2}}{\kappa _{1}}$.

\begin{proposition}
Let $\alpha :I\rightarrow \mathbb{R}^{3}$ be a unit speed Frenet curve with
constant invariants $\tilde{\kappa}_{1}\neq 0$ and $\tilde{\kappa}_{2}\neq 0$%
. Then, the $m_{i}$ $(i=1,2)$ curvatures of $\alpha $ are%
\begin{equation*}
\tilde{\kappa}_{1}=m_{1}^{\prime }\text{ \ \ \ \ \ \ and \ \ \ \ \ \ }\tilde{%
\kappa}_{2}=\dfrac{m_{1}(m_{1}^{\prime }m_{2}-m_{1}m_{2}^{\prime })}{%
m_{1}^{2}+m_{2}^{2}}.
\end{equation*}
\end{proposition}

\begin{proof}
Since $\tilde{\kappa}_{1}=\dfrac{1}{\kappa _{1}(s)}$ and\ $\tilde{\kappa}%
_{2}=\dfrac{\kappa _{2}}{\kappa _{1}}$ then we get%
\begin{equation*}
m_{1}^{\prime }(s)=\left( \dfrac{1}{\kappa _{1}(s)}\right) ^{\prime }=\tilde{%
\kappa}_{1}.
\end{equation*}%
\begin{eqnarray*}
m_{2}^{\prime }(s) &=&\left( \dfrac{1}{\kappa _{1}(s)}\cot \left( \dint
\kappa _{2}(s)ds\right) \right) ^{\prime } \\
&=&m_{1}^{\prime }\dfrac{m_{2}}{m_{1}}-\frac{\kappa _{2}}{\kappa _{1}}\dfrac{%
1}{\sin ^{2}\left( \dint \kappa _{2}(s)ds\right) }
\end{eqnarray*}%
\begin{equation*}
\tilde{\kappa}_{2}=\dfrac{m_{1}(m_{1}^{\prime }m_{2}-m_{1}m_{2}^{\prime })}{%
m_{1}^{2}+m_{2}^{2}}.
\end{equation*}
\end{proof}

\section{The relation between the curve $\protect\alpha $ and focal curve}

Let $\alpha :I\subset \mathbb{R}\rightarrow E^{n}$ be a unit speed Frenet
curve. Suppose that all Euclidean curvatures of the curve $\alpha $ are
nonzero for any $s\in I.$ The curve $C_{\alpha }:I\subset \mathbb{R}%
\rightarrow \mathbb{E}^{n}$ consisting of the centers of the osculating
spheres of the curve $\alpha $ is called the \textit{focal curve }of $\alpha
.$ Then the focal curve $C_{\alpha }$ has a representation:%
\begin{equation}
C_{\alpha }=\alpha
(s)+f_{1}(s)V_{2}(s)+...+f_{n-2}(s)V_{n-1}(s)+f_{n-1}(s)V_{n}(s)  \label{5}
\end{equation}%
where the functions $f_{i}(s),$ $i=1,...,n-1$ are called focal curvature of
the curve $\alpha .$ In \cite{Uribe}$,$ R. Uribe-Vargas found equation
between Euclidean curvatures and focal curvatures. Then by using this
equation we can express the shape curvatures of $\alpha $ in terms of focal
curvatures.

\begin{proposition}
Let $\alpha :I\subset \mathbb{R}\rightarrow E^{n}$ be a space curve with all
Euclidean curvatures different from zero. Then,\newline
\begin{equation}  \label{*}
\end{equation}%
\begin{eqnarray*}
\tilde{\kappa} &=&\frac{d}{ds}\left( \frac{f_{i-2}f_{i-1}f_{i}}{\sqrt{\left(
\left( f_{1}f_{1}^{\prime }+\text{ }f_{2}f_{2}^{\prime }+\cdots
+f_{i-2}f_{i-2}^{\prime }\right) f_{i}\right) ^{2}+\left( \left(
f_{1}f_{1}^{\prime }+\text{ }f_{2}f_{2}^{\prime }+\cdots
+f_{i-1}f_{i-1}^{\prime }\right) f_{i-2}\right) ^{2}}}\right) \text{ ,} \\
\tilde{\kappa}_{j} &=&\frac{f_{i-2}f_{i-1}f_{i}}{f_{j-1}f_{j}}\frac{%
f_{1}f_{1}^{\prime }+\text{ }f_{2}f_{2}^{\prime }+\cdots
+f_{j-1}f_{j-1}^{\prime }}{\sqrt{\left( \left( f_{1}f_{1}^{\prime }+\text{ }%
f_{2}f_{2}^{\prime }+\cdots +f_{i-2}f_{i-2}^{\prime }\right) f_{i}\right)
^{2}+\left( \left( f_{1}f_{1}^{\prime }+\text{ }f_{2}f_{2}^{\prime }+\cdots
+f_{i-1}f_{i-1}^{\prime }\right) f_{i-2}\right) ^{2}}}\text{ ,}
\end{eqnarray*}%
\begin{equation*}
j=1,...,n-1.
\end{equation*}
\end{proposition}

\begin{proof}
According to the first two theorems in \cite{Uribe}, there are relations
between the Frenet curvatures and focal curvatures as follows:

\begin{equation}
\kappa _{i}=\dfrac{f_{1}f_{1}^{\prime }+\text{ }f_{2}f_{2}^{\prime }+\cdots
+f_{i-1}f_{i-1}^{\prime }}{f_{i-1}f_{i}},\text{ \ \ \ \ }i=2,3,...,n-1.
\label{11}
\end{equation}%
By using the Eq. $\left( \ref{10}\right) $ and Eq. $\left( \ref{11}\right) $
we obtain the Eq. $\left( 5.3\right) .$
\end{proof}

\begin{remark}
For $i=1$ $($that is $\sigma _{i}=\sigma ),$ the representation of the shape
curvatures is given with Eq.$\left( \ref{*}\right) $ the same as
representation of the shape curvatures is given with Eq. $(4.2)$ in \cite%
{Rad2}.
\end{remark}

\section{Self-Similar Frenet Curves}

The curve $\alpha :I\subset \mathbb{R}\rightarrow E^{n}$ is called
self-similar if and only if all its invariants $\tilde{\kappa},\tilde{\kappa}%
_{1},...,$ $\tilde{\kappa}_{n-1}$ are constant.

\begin{equation*}
K=\left( 
\begin{array}{cccccccc}
0 & \tilde{\kappa}_{1} & 0 & 0 & \cdots & 0 & 0 & 0 \\ 
-\tilde{\kappa}_{1} & 0 & \tilde{\kappa}_{2} & 0 & \cdots & 0 & 0 & 0 \\ 
0 & -\tilde{\kappa}_{2} & 0 & \tilde{\kappa}_{3} & \cdots & 0 & 0 & 0 \\ 
\vdots & \vdots & \vdots & \vdots & \ddots & \vdots & \vdots & \vdots \\ 
0 & 0 & 0 & 0 & \cdots & -\tilde{\kappa}_{n-2} & 0 & \tilde{\kappa}_{n-1} \\ 
0 & 0 & 0 & 0 & \cdots & 0 & -\tilde{\kappa}_{n-1} & 0%
\end{array}%
\right)
\end{equation*}%
According to \cite{Greub}, the normal form of the matrix $K$ is given
following as 
\begin{equation*}
\left( 
\begin{array}{cccccccc}
0 & \lambda _{1} & 0 & 0 & \cdots & 0 & 0 & 0 \\ 
-\lambda _{1} & 0 & 0 & 0 & \cdots & 0 & 0 & 0 \\ 
0 & 0 & 0 & \lambda _{2} & \cdots & 0 & 0 & 0 \\ 
0 & 0 & -\lambda _{2} & 0 & \cdots & 0 & 0 & 0 \\ 
\vdots & \vdots & \vdots & \vdots & \ddots & \vdots & \vdots & \vdots \\ 
0 & 0 & 0 & 0 & \cdots & 0 & 0 & \lambda _{m} \\ 
0 & 0 & 0 & 0 & \cdots & 0 & -\lambda _{m} & 0%
\end{array}%
\right)
\end{equation*}%
Then the symetric matrix $K^{2}$ has $m=\dfrac{n}{2}$ negative eigenvalues
with multiplicity two: $-\lambda _{1}^{2},$ $-\lambda _{2}^{2},$...,$%
-\lambda _{m}^{2}.$

\subsection{Self- Similar curves in Even-Dimensional Euclidean Spaces \ \ \
\ \ \ \ \ \ \ \ \ \ \ \ \ \ \ \ \ \ \ \ \ \ \ \ \ \ \ \ \ }

In this section, we deal with self-similar curves. We can write a curve if
we have constant shape curvatures of the curve. Besides, the curve $\alpha $
can be written arc length parameter $\sigma _{i}$ of its $V_{i}-$indicatrix
curve.

\begin{theorem}
Let $\alpha :I\ \rightarrow \mathbb{R}^{2m}\ $be a self-similar curve. The
curve $\alpha $ parameterized by the general arc-length parameter $\sigma
_{i}$ of its $V_{i}-$indicatrix curve. The curve $\alpha $ can be written,%
\begin{equation}
\alpha \left( \sigma _{i}\right) =\left( \frac{a_{1}}{b_{1}}e^{\tilde{\kappa}%
\sigma _{i}}\sin \theta _{1},-\frac{a_{1}}{b_{1}}e^{\tilde{\kappa}\sigma
_{i}}\cos \theta _{1},\cdots ,\frac{a_{m}}{b_{m}}e^{\tilde{\kappa}\sigma
_{i}}\sin \theta _{m},-\frac{a_{m}}{b_{m}}e^{\tilde{\kappa}\sigma _{i}}\cos
\theta _{m}\right)  \label{12}
\end{equation}%
where $b_{j}=\sqrt{\tilde{\kappa}^{2}+\lambda _{j}^{2}}$ $\ \ \ ,$ $\ \
\theta _{j}=\sqrt{\lambda _{j}^{2}\sigma _{i}}+\arccos \left( \dfrac{\sqrt{%
\lambda _{j}^{2}}}{\sqrt{\tilde{\kappa}^{2}+\lambda _{j}^{2}}}\right) $ for $%
j=1,2,\cdots ,m$ and the real different nonzero numbers $a_{1},a_{2},\cdots
,a_{m}$ are solution of the system%
\begin{equation*}
\langle V_{j},V_{j}\rangle =1,\ \ j=1,2,\cdots ,m.
\end{equation*}
\end{theorem}

\begin{proof}
We express in the form of a column vector $\omega (\sigma
_{i})=(V_{1}(\sigma _{i}),V_{2}(\sigma _{i}),\cdots ,V_{2m-1}(\sigma
_{i}),V_{2m}(\sigma _{i}))$ of unit vector fields $\left \{ V_{1}(\sigma
_{i}),V_{2}(\sigma _{i}),\cdots ,V_{2m-1}(\sigma _{i}),V_{2m}(\sigma
_{i})\right \} $. From the solution of the differential equation $\dfrac{d}{%
d\sigma _{i}}\omega ={\small K}\omega ,$ we can calculate the unit vector
fields $\left \{ V_{1}(\sigma _{i}),V_{2}(\sigma _{i}),\cdots
,V_{2m-1}(\sigma _{i}),V_{2m}(\sigma _{i})\right \} $. Also, the first unit
vector has%
\begin{equation*}
V_{1}(\sigma _{i})=(a_{1}\cos (\lambda _{1}\sigma _{i}),a_{1}\sin (\lambda
_{1}\sigma _{i}),\cdots a_{m}\cos (\lambda _{m}\sigma _{i}),a_{_{m}}\sin
(\lambda _{m}\sigma _{i})).
\end{equation*}%
Because of $\langle V_{1},V_{1}\rangle =1,$ we obtain that $\tsum
\limits_{j=1}^{m}(a_{j})^{2}=1.$

The parametric equation of the curve $\alpha $ is given by $X=(x_{1}(\sigma
),x_{2}(\sigma ),\cdots ,x_{2m-1}(\sigma ),x_{2m}(\sigma )).$ Then the Eq. $%
\left( \ref{6}\right) $ can write $\dfrac{d}{d\sigma _{i}}X=\dfrac{1}{\sqrt{%
\kappa _{i-1}^{2}+\kappa _{i}^{2}}}V_{1}.$ We can see easily that $\sqrt{%
\kappa _{i-1}^{2}+\kappa _{i}^{2}}=e^{-\tilde{\kappa}\sigma _{i}}$. Hence,
we have%
\begin{equation*}
\dfrac{d}{d\sigma _{i}}x_{2j-1}=a_{j}e^{\tilde{\kappa}\sigma }\cos (\lambda
_{j}\sigma _{i})\text{ \ \ \ and \ \ \ \ }\dfrac{d}{d\sigma _{i}}%
x_{2j}=a_{j}e^{\tilde{\kappa}\sigma _{i}}\sin (\lambda _{j}\sigma _{i}),%
\text{ \ \ for \ \ }j=1,2,\cdots ,m.
\end{equation*}
Integrating the last equations, we obtain%
\begin{equation}
x_{2j-1}=\frac{a_{j}}{\tilde{\kappa}}e^{^{\tilde{\kappa}_{1}\sigma
_{i}}}\cos (\lambda _{j}\sigma _{i})+\frac{\lambda _{j}}{\tilde{\kappa}}%
x_{2j}  \label{13}
\end{equation}%
\begin{equation}
x_{2j}=\frac{a_{j}}{\tilde{\kappa}}e^{^{\tilde{\kappa}_{1}\sigma _{i}}}\sin
(\lambda _{j}\sigma _{i})-\frac{\lambda _{i}}{\tilde{\kappa}}x_{\substack{ %
2j-1  \\ .}}.  \label{14}
\end{equation}%
Using the Eq.$\left( \ref{13}\right) $ and Eq.$\left( \ref{14}\right) ,$ we
get two equations:%
\begin{eqnarray*}
x_{2j-1} &=&\frac{a_{j}}{\ \tilde{\kappa}^{2}+\lambda _{j}^{2}}e^{^{\tilde{%
\kappa}\sigma _{i}}}(\tilde{\kappa}\cos (\lambda _{j}\sigma _{i})+\lambda
_{j}\sin (\lambda _{j}\sigma _{i}))=\frac{a_{j}}{b_{j}}e^{\tilde{\kappa}%
\sigma _{i}}\sin \theta _{j}, \\
x_{2j} &=&\frac{a_{j}}{\ \tilde{\kappa}^{2}+\lambda _{j}^{2}}e^{^{\tilde{%
\kappa}\sigma _{i}}}(\tilde{\kappa}\sin (\lambda _{j}\sigma _{i})-\lambda
_{j}\cos (\lambda _{j}\sigma _{i}))=-\frac{a_{j}}{b_{j}}e^{\tilde{\kappa}%
\sigma _{i}}\cos \theta _{j}
\end{eqnarray*}%
where%
\begin{eqnarray*}
\sin \theta _{j} &=&\dfrac{b_{j}}{\tilde{\kappa}}(\cos (\lambda _{j}\sigma
_{i})-\dfrac{\lambda _{j}}{b_{j}}\cos \theta _{j}), \\
\cos \theta _{j} &=&-\dfrac{b_{j}}{\tilde{\kappa}}(\sin (\lambda _{j}\sigma
_{i})-\dfrac{\lambda _{j}}{b_{j}}\sin \theta _{j}).
\end{eqnarray*}%
From here, $\theta _{j}$ can be found%
\begin{equation*}
\theta _{j}=\lambda _{j}\sigma _{i}+\arccos \frac{\lambda _{j}}{\sqrt{\tilde{%
\kappa}^{2}+\lambda _{j}^{2}}}.
\end{equation*}%
We can write from the equation $\dfrac{d}{d\sigma _{i}}\varepsilon =${\small 
$K$}$\varepsilon $%
\begin{eqnarray*}
V_{1}(\sigma _{i}) &=&e^{-\tilde{\kappa}\sigma _{i}}\frac{d}{d\sigma _{i}}%
\alpha (\sigma _{i}) \\
V_{2}(\sigma _{i}) &=&\frac{1}{\tilde{\kappa}_{1}}\frac{d}{d\sigma _{i}}%
V_{1}(\sigma _{i}) \\
V_{3}(\sigma _{i}) &=&\frac{1}{\tilde{\kappa}_{2}}(-\tilde{k}%
_{1}V_{1}(\sigma _{i})+\frac{d}{d\sigma _{i}}V_{2}(\sigma _{i})) \\
V_{4}(\sigma _{i}) &=&\frac{1}{\tilde{\kappa}_{3}}(\tilde{\kappa}%
_{2}V_{2}(\sigma _{i})+\frac{d}{d\sigma _{i}}V_{3}(\sigma _{i})) \\
&&\vdots \\
V_{m}(\sigma _{i}) &=&\frac{1}{\tilde{\kappa}_{m-1}}(\tilde{\kappa}%
_{m-2}V_{m-2}(\sigma _{i})+\frac{d}{d\sigma _{i}}V_{m-1}(\sigma _{i})).
\end{eqnarray*}%
So, we show that by using an algebric calculus 
\begin{eqnarray*}
\langle V_{1},V_{1}\rangle &=&1\Rightarrow \tsum \limits_{j=1}^{m}a_{j}^{2}=1
\\
\langle V_{2},V_{2}\rangle &=&1\Rightarrow \tsum
\limits_{j=1}^{m}a_{j}^{2}\lambda _{j}^{2}=\tilde{\kappa}_{1}^{2} \\
\langle V_{3},V_{3}\rangle &=&1\Rightarrow \tsum
\limits_{j=1}^{m}a_{j}^{2}(1-\lambda _{j}^{2})^{2}=\tilde{\kappa}_{2}^{2} \\
\langle V_{4},V_{4}\rangle &=&1\Rightarrow \tsum
\limits_{j=1}^{m}a_{j}^{2}\lambda _{j}^{2}(\tilde{\kappa}_{2}^{2}-\lambda
_{j}^{2})^{2}=\tilde{\kappa}_{2}^{2}\tilde{\kappa}_{3}^{2}\text{ \ \ and so
on.}
\end{eqnarray*}%
The proof is completed.
\end{proof}

\begin{remark}
For $i=1$ $($that is $\sigma _{i}=\sigma ),$ the representation of the curve 
$\alpha $ in even-dimensional Euclidean spaces is given with Eq.$\left( \ref%
{12}\right) $ the same as representation of the curve $\alpha $ is given
with Eq. $(5.2)$ in \cite{Rad2}.
\end{remark}

\subsection{Self-Similar Curves in Odd-Dimensional Euclidean Spaces \ \ \ \
\ \ \ \ \ \ \ \ \ \ \ \ \ \ \ \ \ \ \ \ \ \ \ \ \ \ \ \ \ \ \ \ \ \ \ \ \ \
\ }

We can also express self-similar curves in $\mathbb{R}^{2m+1}.$

\begin{theorem}
Let $\alpha :I\ \rightarrow \mathbb{R}^{2m+1}\ $be a self-similar curve. The
curve $\alpha $ parameterized by the general arc-length parameter $\sigma
_{i}$ of its $V_{i}-$indicatrix curve. The curve $\alpha $ can be written,
\end{theorem}

\begin{equation}
\alpha \left( \sigma _{i}\right) =\left( \frac{a_{1}}{b_{1}}e^{\tilde{\kappa}%
\sigma _{i}}\sin \theta _{1},-\frac{a_{1}}{b_{1}}e^{\tilde{\kappa}\sigma
_{i}}\cos \theta _{1},\cdots ,\frac{a_{m}}{b_{m}}e^{\tilde{\kappa}\sigma
_{i}}\sin \theta _{m},-\frac{a_{m}}{b_{m}}e^{\tilde{\kappa}\sigma _{i}}\cos
\theta _{m},a_{m+1}e^{\tilde{\kappa}\sigma _{i}}\right)  \label{15}
\end{equation}%
where $b_{j}=\sqrt{\tilde{\kappa}^{2}+\lambda _{j}^{2}}$ $\ \ \ ,$ $\ \
\theta _{j}=\sqrt{\lambda _{j}^{2}\sigma _{i}}+\arccos \left( \dfrac{\sqrt{%
\lambda _{j}^{2}}}{\sqrt{\tilde{\kappa}^{2}+\lambda _{j}^{2}}}\right) $ for $%
j=1,2,\cdots ,m$ and the real different nonzero numbers $a_{1},a_{2},\cdots
,a_{m+1}.$ To calculate $a_{1},a_{2},\cdots ,a_{m+1}$ we use these $m+1$
equations%
\begin{eqnarray*}
V_{1}(\sigma _{i}) &=&e^{-\tilde{\kappa}\sigma _{i}}\frac{d}{d\sigma _{i}}%
\alpha (\sigma _{i}) \\
V_{2}(\sigma _{i}) &=&\frac{1}{\tilde{\kappa}_{1}}\frac{d}{d\sigma _{i}}%
V_{1}(\sigma _{i}) \\
V_{3}(\sigma _{i}) &=&\frac{1}{\tilde{\kappa}_{2}}(-\tilde{\kappa}%
_{1}V_{1}(\sigma _{i})+\frac{d}{d\sigma _{i}}V_{2}(\sigma _{i})) \\
V_{4}(\sigma _{i}) &=&\frac{1}{\tilde{\kappa}_{3}}(\tilde{\kappa}%
_{2}V_{2}(\sigma _{i})+\frac{d}{d\sigma _{i}}V_{3}(\sigma _{i})) \\
&&\vdots \\
V_{m+1}(\sigma _{i}) &=&\frac{1}{\tilde{\kappa}_{m}}(\tilde{\kappa}%
_{m-1}V_{m-1}(\sigma _{i})+\frac{d}{d\sigma _{i}}V_{m}(\sigma _{i})).
\end{eqnarray*}%
\begin{equation*}
\langle V_{j},V_{j}\rangle =1,\ \ j=1,2,\cdots ,m+1.
\end{equation*}%
From here%
\begin{eqnarray*}
\langle V_{1},V_{1}\rangle &=&1\Rightarrow \tsum \limits_{i=1}^{m}a_{i}^{2}+%
\tilde{\kappa}^{2}a_{m+1}^{2}=1 \\
\langle V_{2},V_{2}\rangle &=&1\Rightarrow \tsum
\limits_{i=1}^{m+1}a_{i}^{2}\lambda _{i}^{2}=\tilde{\kappa}_{1}^{2} \\
\langle V_{3},V_{3}\rangle &=&1\Rightarrow \tsum
\limits_{i=1}^{m}a_{i}^{2}(1-\frac{\lambda _{i}^{2}}{\tilde{\kappa}_{1}^{2}}%
)^{2}+\tilde{\kappa}^{2}a_{m+1}^{2}=\frac{\tilde{\kappa}_{2}^{2}}{\tilde{%
\kappa}_{1}^{2}} \\
\langle V_{i},V_{i}\rangle &=&1\ ,i=4,\ldots ,m+1,
\end{eqnarray*}

\begin{proof}
The proof is the same as the proof of Theorem 2.
\end{proof}

\begin{remark}
For $i=1$ $($that is $\sigma _{i}=\sigma ),$ the representation of the curve 
$\alpha $ in odd-dimensional Euclidean spaces is given with Eq.$\left( \ref%
{12}\right) $ the same as representation of the curve $\alpha $ is given
with Eq. $(5.2)$ in \cite{Rad2}.
\end{remark}

\begin{example}
Let be shape curvature functions $\tilde{\kappa}_{1}=\frac{3}{\sqrt{13}}$
and $\tilde{\kappa}_{2}=\frac{2}{\sqrt{13}}$ of the curve $\alpha $ in $%
\mathbb{R}^{3}.$ We can calculate the curve $\alpha $ (see Figure 2)
corresponding to them%
\begin{equation*}
\alpha \left( \sigma _{2}\right) =\left( \frac{a_{1}}{b_{1}}e^{\tilde{\kappa}%
\sigma _{2}}\sin \theta _{1},-\frac{a_{1}}{b_{1}}e^{\tilde{\kappa}\sigma
_{2}}\cos \theta _{1},a_{2}e^{\tilde{\kappa}\sigma _{2}}\right) 
\end{equation*}%
where $\sigma _{2}$ is arc length parameter of its normal indicatrix curve.
Then we use equations in odd dimensional space, an algebraic calculus shows
that%
\begin{equation*}
\lambda _{1}^{2}=\frac{5}{13},\text{ }a_{1}=\frac{3}{\sqrt{5}}\text{ and }%
a_{2}=\frac{2}{\sqrt{5}},\text{ }b_{1}=\frac{\sqrt{36+5s^{4}}}{\sqrt{13}s^{2}%
},\text{ \ }\theta _{1}=\sqrt{\frac{5}{13}}\sigma _{2}+\arccos \left( \frac{%
\sqrt{\frac{5}{13}}}{\sqrt{\frac{36}{13}\frac{1}{s^{4}}+\frac{5}{13}}}%
\right) .
\end{equation*}%
\begin{equation*}
\FRAME{itbpF}{2.93in}{2.22in}{0in}{}{}{Figure}{\special{language "Scientific
Word";type "GRAPHIC";maintain-aspect-ratio TRUE;display "USEDEF";valid_file
"T";width 2.93in;height 2.22in;depth 0in;original-width
3.7498in;original-height 2.8331in;cropleft "0";croptop "1";cropright
"1";cropbottom "0";tempfilename 'N20AIP00.wmf';tempfile-properties "XPR";}}
\end{equation*}%
\begin{equation*}
\text{Figure 2. The curve }\alpha .
\end{equation*}
\end{example}

\section{Geometric Interpretation of Shape Curvatures in 3-Euclidean Spaces}

In this section, we mentioned that geodesic curvatures of indicatrix curves
are related to shape curvatures of the curve.

\begin{proposition}
Geodesic curvatures of indicatrix curves are invariant under the group $%
Sim^{+}(\mathbb{R}^{3}).$
\end{proposition}

\begin{proof}
Let $%
\begin{array}{ccccc}
\alpha & : & I & \rightarrow & \mathbb{R}^{3}%
\end{array}%
$ be a Frenet curve with $\left\{ V_{1},V_{2},V_{3},\kappa _{1},\kappa
_{2}\right\} $ Frenet apparatus and $\tilde{\kappa}_{1},$ $\tilde{\kappa}%
_{2} $ be shape curvatures$.$ Its indicatrix curve show that $%
\begin{array}{ccccc}
\gamma & : & I & \rightarrow & S^{2}%
\end{array}%
$ is a spherical curve with arc length parameter $\sigma _{i}$ ($i=1,2,3$).
of $\gamma .$ Let us denote $t(\sigma _{i})=\frac{d}{d\sigma _{i}}\gamma
(\sigma _{i})$ and we call $t(\sigma _{i})$ a unit tangent vector of $\gamma 
$. We now set a vector $\rho (\sigma _{i})=\gamma (\sigma _{i})\wedge
t(\sigma _{i})$ along the curve $\gamma $. This frame is called the Sabban
frame of $\gamma $ on $S^{2}$. Then we have the following spherical Frenet
formulae of $\gamma $%
\begin{equation}
\frac{d}{d\sigma _{i}}\left[ 
\begin{array}{c}
\gamma \\ 
t \\ 
\rho%
\end{array}%
\right] =\left[ 
\begin{array}{ccc}
0 & 1 & 0 \\ 
-1 & 0 & \kappa _{g} \\ 
0 & -\kappa _{g} & 0%
\end{array}%
\right] \left[ 
\begin{array}{c}
\gamma \\ 
t \\ 
\rho%
\end{array}%
\right]  \label{16}
\end{equation}%
where $\kappa _{g}=\det (\gamma ,t,\dfrac{dt}{d\sigma _{i}})$ is the
geodesic curvature of $\gamma $ at $\gamma (\sigma _{i}).$

Let $\bar{\alpha}$ be similar curve of $\alpha $ under the group $Sim^{+}(%
\mathbb{R}^{3})$ and $\left\{ \bar{V}_{1},\bar{V}_{2},\bar{V}_{3},\bar{\kappa%
}_{1},\bar{\kappa}_{2},\widetilde{\bar{\kappa}}_{1},\widetilde{\bar{\kappa}}%
_{2}\right\} $ be Frenet apparatus of the curve $\bar{\alpha}.$ Its
indicatrix curve $%
\begin{array}{ccccc}
\bar{\gamma} & : & I & \rightarrow & S^{2}%
\end{array}%
$ is a spherical curve with arc length parameter $\bar{\sigma}_{i}$ ($%
i=1,2,3 $). of $\bar{\gamma}.$ The orthogonal frame $\left\{ \bar{\gamma}(%
\bar{\sigma}_{i}),\bar{t}(\bar{\sigma}_{i}),\bar{\rho}(\bar{\sigma}%
_{i})\right\} $ along $\bar{\gamma}$ is called the Sabban frame of $\bar{%
\gamma}.$

If the curve $\gamma $ is a tangent indicatrix curve with arc length
parameter $\sigma _{1}$ then the following Frenet-Serret formulas hold%
\begin{equation*}
t(\sigma _{1})=\dfrac{d}{d\sigma _{1}}\gamma (\sigma _{1})
\end{equation*}%
\begin{equation*}
\dfrac{d}{d\sigma _{1}}t(\sigma _{1})=-\gamma (\sigma _{1})+\dfrac{\tilde{%
\kappa}_{2}}{\tilde{\kappa}_{1}}\rho (\sigma _{1}).
\end{equation*}%
So, the geodesic curvature $\kappa _{g}=\dfrac{\tilde{\kappa}_{2}}{\tilde{%
\kappa}_{1}}.$ Similarly, we calculate $\bar{\kappa}_{g}=\dfrac{\widetilde{%
\bar{\kappa}}_{2}}{\widetilde{\bar{\kappa}}_{1}}.$ Using the Eq.$\left( \ref%
{Prop.1}\right) ,$ we get $\bar{\kappa}_{g}=\kappa _{g}.$

If the curve $\gamma $ is a normal indicatrix curve with arc length
parameter $\sigma _{2}$ then the following Frenet-Serret formulas hold

\begin{eqnarray*}
t(\sigma _{2}) &=&\dfrac{d}{d\sigma _{2}}\gamma (\sigma _{2}) \\
\dfrac{d}{d\sigma _{2}}t(\sigma _{2}) &=&-\gamma (\sigma _{2})+\tilde{\kappa}%
_{1}^{2}\frac{d}{d\sigma _{2}}\left( \frac{\tilde{\kappa}_{2}}{\tilde{\kappa}%
_{1}}\right) \rho (\sigma _{2}).
\end{eqnarray*}

$\kappa _{g}(\sigma _{2})=\tilde{\kappa}_{1}^{2}\dfrac{d}{d\sigma _{2}}%
\left( \dfrac{\tilde{\kappa}_{2}}{\tilde{\kappa}_{1}}\right) $. Similarly,
we get $\bar{\kappa}_{g}=\widetilde{\bar{\kappa}_{1}}^{2}\dfrac{d}{d%
\overline{\sigma _{2}}}(\dfrac{\widetilde{\bar{\kappa}_{2}}}{\widetilde{\bar{%
\kappa}_{1}}}).$ So, $\bar{\kappa}_{g}$ is equal to $\kappa _{g}.$

If the curve $\gamma $ is a binormal indicatrix curve with arc length
parameter $\sigma _{3}$ then the following Frenet-Serret formulas hold%
\begin{eqnarray*}
t(\sigma _{3}) &=&\dfrac{d}{d\sigma _{3}}\gamma (\sigma _{3}) \\
\dfrac{d}{d\sigma _{3}}t(\sigma _{3}) &=&-\gamma (\sigma _{3})+\frac{\tilde{%
\kappa}_{1}}{\tilde{\kappa}_{2}}\rho (\sigma _{2})
\end{eqnarray*}%
$\kappa _{g}(\sigma _{2})=\dfrac{\tilde{\kappa}_{1}}{\tilde{\kappa}_{2}}$.
Similarly, we get $\bar{\kappa}_{g}=\dfrac{\widetilde{\bar{\kappa}_{1}}}{%
\widetilde{\bar{\kappa}_{2}}}.$ So, $\kappa _{g}$ is equal to $\bar{\kappa}%
_{g}.$

The proof is completed.\newpage
\end{proof}

\end{document}